\documentclass[12pt,oneside,reqno]{amsart}
\usepackage[all]{xy}
\usepackage{amsfonts,amsmath,oldgerm,amssymb,amscd}
\UseComputerModernTips
\numberwithin{equation}{section}
\usepackage[breaklinks]{hyperref}
\allowdisplaybreaks

\usepackage{enumerate}
\usepackage{caption} 
\usepackage{enumitem}   
\usepackage{cleveref}
\usepackage{comment}

\newcommand{\PP}{{\mathbb{P}}}
\newtheorem{theorem}{Theorem}[section]
\newtheorem{proposition}[theorem]{Proposition}
\newtheorem{lemma}[theorem]{Lemma}
\newtheorem{coro}[theorem]{Corollary}
\newtheorem{conjecture}[theorem]{Conjecture}
\theoremstyle{definition}

\newcommand{\del}{\delta}

\newcommand{\Q}{\mbox{$\mathbb Q$}}
\newcommand{\R}{\mbox{$\mathbb R$}}     % For Real numbers
\newcommand{\C}{\mbox{$\mathbb C$}}     %For Complex numbers

\newcommand{\hhat}{h}  

\begin{document}

	\title[Quantitative bounds on integrality for  post-critically finite maps]{Quantitative bounds on integrality for  post-critically finite maps} %\\ \today}

\author[R. Padhy]{R. Padhy}
\address{Rudranarayan Padhy, Department of Mathematics, National Institute of Technology Calicut, 
	Kozhikode-673 601, India.}
\email{rudranarayan\_p230169ma@nitc.ac.in; padhyrudranarayan1996@gmail.com}

\author[S. S. Rout]{S. S. Rout}
\address{Sudhansu Sekhar Rout, Department of Mathematics, National Institute of Technology Calicut, 
	Kozhikode-673 601, India.}
\email{sudhansu@nitc.ac.in; lbs.sudhansu@gmail.com}


\begin{abstract}
Let $K$ be a number field with algebraic closure $\overline{K}$ and let $S$ be a finite set of places of $K$ that contain all the archimedean places. For an integer $d \ge 2$, consider the unicritical polynomial family $f_{d,c}(z) = z^d + c$. Recently, Benedetto and Ih studied the distribution of post-critically finite parameters $c$ that are $S$-integral relative to a fixed point $\alpha \in \overline{K}$ such that $f_{d, \alpha}$ is not post-critically finite. In this paper, we study the quantitative aspects of their result. In particular, under some additional assumptions we establish quantitative bounds on the number of $S$-integral post-critically finite parameters in the generalized Mandelbrot set $\mathcal{M}_{d, v}$ relative to a non post-critically finite parameter $\alpha$ as $\alpha$ varies over number fields of bounded degree.  
\end{abstract}

\maketitle
\pagenumbering{arabic}
\pagestyle{headings}

\section{Introduction}
Let $K$ be a field with algebraic closure $\overline{K}$ and let $f(z) \in K(z)$ be a rational function of degree $d\geq 2$ defined over $K$. We define the $n$-fold composition of $f$ with itself by $f^{n}(z) := f \circ f\circ \cdots \circ f(z)$ with $f^0(z) := z$. The forward orbit of a point $z \in \mathbb{P}^1(\overline{K})$ is the set $\{ f^n(z) : n \ge 0 \}$ and the strict forward orbit of $z$ is the set $\{ f^n(z) : n \ge 1 \}$. We say that $z$ is {\em periodic} under $f$ if $f^n(z) = z$ for some $n \ge 1$. Similarly, $z$ is {\em preperiodic} under $f$ if there exist integers $n > m \ge 0$ such that $f^n(z) = f^m(z)$. Again $z$ is {\em strictly preperiodic} if it is preperiodic but not periodic. We say that $f$ is {\em post-critically finite}  if the orbit of the critical points of $f$ under iteration is finite. Post-critically finite maps play a fundamental role in complex and arithmetic dynamics, because many dynamical features of $f$ can be obtained from the behavior of the critical points under iteration (see, for example, \cite{BD13, BFHJY, Benedetto24, Benedetto2014, Buff18, Eps12}). One broad class of examples of post-critically finite maps is the flexible Latt\`es maps, obtained by descending an endomorphism of a torus to the projective line. Other than these examples, post-critically finite maps are relatively sparse.

Consider post-critically finite parameters in the one-parameter family of unicritical polynomials  
\[f_{d,c}(z) := z^d + c,\] 
where $d \geq 2$ is an integer and $c\in \mathbb{P}^1(\overline{K})$. Observe that the polynomial $f_{d,c}$ has critical points at $z = 0, \infty$. Since $\infty$ is fixed, it follows that $f_{d,c}$ is post-critically finite if and only if the forward orbit of the critical point $z = 0$,
$$\{f^n_{d,c}(0) : n \geq 0\}$$
is a finite set. Since $c$ is a root of the polynomial $f^n_{d,c}(0)-f^m_{d,c}(0)$  for some integers $n > m \geq 0$, so post-critically finite parameter $c\in \overline{\Q}$. In fact, $c$ is an algebraic integer. For any fixed integers $d \geq 2$ and $B \geq 1$, it is known that there are only finitely many conjugacy classes of post-critically finite rational maps of degree $d$ which can be defined over a number field of degree at most $B$, except (when $d$ is a perfect square) for the infinite family of flexible Latt\`es maps (see \cite{Benedetto2014}).  In particular, for any number field $K$, there are only finitely many $c\in K$ for which $f_{d,c}$ is post-critically finite. For example, the only quadratic polynomials of the form $z^2 + c$ with $c$ a rational number which are post-critically finite are $z^2, z^2-1$ and $z^2-2$ and every quadratic polynomial is conjugate to a polynomial of this form.

Let $M_K$ be the set of places of $K$ and $S$ be a finite subset of $M_K$ including all the archimedean places of $K$. Let $K_v$ be the completion of $K$ at a place $v\in M_K$, with absolute value $|\cdot|_v$ and $\mathbb{C}_v$ be the completion of an algebraic closure of $K_v$ with absolute value $|\cdot|_v$. Suppose that $\alpha, \beta \in \overline{K}$.  We say that $\beta$ is \emph{$S$-integral relative to} $\alpha$ if no conjugate of $\beta$ meets any conjugate of $\alpha$ at the primes lying outside of $S$. More precisely, for every place $v \notin S$ and every pair of $K$-embeddings $\sigma:\overline{K}\to \mathbb{C}_v$ and $\tau:\overline{K}\to \mathbb{C}_v$, the points $\sigma(x)$ and $\tau(y)$ lie in different residue classes of $\mathbb{P}^1(\mathbb{C}_v)$, that is, 
\begin{equation*}
	\begin{cases}
		|\sigma(x) - \tau(y)|_v \geq 1 \quad &\text{if } |\tau(y)|_v \leq 1, \\
		|\sigma(x)|_v \leq 1 \quad &\text{if } |\tau(y)|_v > 1.
	\end{cases}
\end{equation*}

Benedetto and Ih conjectured a finiteness property for $S$-integral parameters $c\in \overline{K}$ for which $f_{d, c}$ is post-critically finite in the context of dynamical systems (see \cite[Conjecture 1.2]{Benedetto}).
\begin{conjecture} \label{conj1}
    %Let $K$ be a number field with algebraic closure $\overline{K}$, let $S\subseteq M_K$ be a finite set of places of $K$ including all the archimedean places, let $d\geq 2$ be an integer and for any $c\in \overline{K}$, let $f_{d,c}:=z^d+c$. 
Let $\alpha\in \overline{K}$ and suppose that $f_{d, \alpha}$ is not post-critically finite. Then there are finitely many parameters $c\in \overline{K}$ that are $S$-integral relative to $(\alpha)$, for which $f_{d,c}$ is post-critically finite.
\end{conjecture}

For the family $f_{d,c}$, the parameter $c$ lives in a moduli space isomorphic to $\mathbb{A}^1$ and the values of $c$ for which $f_{d,c}$ is post-critically finite may be considered as special points on this variety, analogous to torsion points on an abelian variety or CM points on a modular curve. From this perspective, Conjecture~\ref{conj1} describes the integrality of these (dynamically) special points relative to both special and non-special points on this moduli space (see \cite{baker2008, HabIh}).

In \cite{Benedetto}, the authors proved Conjecture \ref{conj1} under some additional assumptions. We need the following notation to state their result. For each integer $d\geq 2$ and for each place $v\in M_K$, the family $f_{d, c}$ has an associated {\em $v$-adic generalized Mandelbrot set} $\mathcal{M}_{d,v}$ or {\em multibrot set}, defined by
\begin{equation} \label{Mandelbrot}
    \mathcal{M}_{d,v}:= \{c\in \mathbb{C}_v: \mbox{the orbit}\;\; \{f_{d,c}^n(0):n\geq 0\}\;\; \mbox{is bounded}\}.
\end{equation}
Observe that if $c\in \mathbb{C}_v$ is a post-critically finite parameter for $f_{d,c}$, then clearly $c\in \mathcal{M}_{d,v}$. Furthermore, if $v$ is an archimedean place, then $\mathcal{M}_{d,v}$ is the set of parameters $c\in \mathbb{C}$ for which the Julia set of $f_{d,c}$ is connected. In the archimedean case, the Mandelbrot set is compact and the parameters $c$ for which $z = 0$ is strictly preperiodic are called {\em Misiurewicz parameters} and they are dense in the boundary $\partial \mathcal{M}_{d, v}$ of the Mandelbrot set $\mathcal{M}_{d,v}$. The parameters for which $z = 0$ is periodic lie in the interior of $\mathcal{M}_{d,v}$ but also accumulate on $\partial \mathcal{M}_{d,v}$.  Similar phenomena occur in the non-archimedean setting; for example, Rivera-Letelier described analogous Misiurewicz bifurcations in the $p$-adic case (see \cite{Letelier}). The following result is a finiteness property of post-critically finite maps on moduli spaces from the view point of Diophantine geometry (see for example \cite{baker2008}).

\begin{theorem}[\cite{Benedetto}]\label{thm1}
Assume the hypothesis of Conjecture \ref{conj1}. Suppose that for every archimedean place $v$ of $K$ and for every $K$-embeddings $\tau:K(\alpha)\to \mathbb{C}_v$, the image $\tau(\alpha)$ does not lie in the boundary $\partial \mathcal{M}_{d,v}$ of the $v$-adic generalized Mandelbrot set $\mathcal{M}_{d,v}$. Then 
Conjecture \ref{conj1} holds.
\end{theorem}

 Recently, we have seen the effectiveness of quantitative equidistribution techniques in answering questions in unlikely intersections, particularly with a view toward a uniform result (see for instance, \cite{demarco2020, demarco2022, Prout, yap}). In this paper, we provide quantitative bounds on the size of Galois orbits of post-critically finite parameters. For an algebraic number $x\in \overline{\Q}$ and a number field $K$, we let $\mathcal{G}(x):=\mathrm{Gal}(\overline{K}/K)\cdot x$ denote the $\mathrm{Gal}(\overline{K}/K)$-orbit of $x$ and $|\mathcal{G}(x)|$ denote the size of the set $\mathcal{G}(x)$.  
\begin{theorem}\label{thm2}
Let $K$ be a number field with its algebraic closure $\overline{K}$ and $S$ be a finite set of places of $K$ including all the archimedean places. Let $f_{d,c}(z):=z^d+c$ be the unicritical family of degree $d\ge 2$ defined over $\overline{K}$, suppose that 
\begin{enumerate}[label={\upshape(\alph*)}]
\item  $f_{d,\alpha}$ is not post-critically finite, 
\item  for every archimedean place  $v$ of  $K$ and every $K$-embedding
$\tau : K(\alpha) \hookrightarrow \mathbb{C}_v$, the image $\tau(\alpha)$ doesn't lie in the boundary $\partial \mathcal{M}_{d,v}$, i.e., $\tau(\alpha) \notin \partial \mathcal{M}_{d,v}$.
\end{enumerate}
 Then for any number field $L/K$, there exists a constant $C_0 = C_0(L, |S|, d)$ such that for any $\alpha \in L$ satisfying $(a)$ and $(b)$, if $c \in \mathbb{Q}$ is a post-critically finite parameter for $f_{d,c}$ which is $S$-integral relative to $\alpha$, then $|\mathcal{G}(c)| < C_0$.
\end{theorem}
In the next result, we will provide an upper bound that grows exponentially.  
\begin{theorem}\label{thm02} 
	Let $K$ be a number field with its algebraic closure $\overline{K}$, $S$ be a finite set of places of $K$ including all the archimedean places. Let $f_{d,c}(z):=z^d+c$ be the unicritical family of degree $d\ge 2$ defined over $\overline{K}$. Then for any number field $L/K$, there exists a constant $C_1=C_1(d) >0$, independent of $L$ and $S$ such that for all $\alpha \in L$ satisfying $(a)$ and $(b)$ given in Theorem \ref{thm2}, the set 
	$$\{c \in \overline{\mathbb{Q}}: |\mathcal{G}(c)|  >C_1 |S|^3 [L:\Q]^{8}, c \text{ is $S$-integral relative to $\alpha$}\}$$ is a union of at most $|S| ~\mathrm{Gal}(\overline{K}/K)$-orbits.
\end{theorem}

The plan of our paper is as follows. In Section \ref{sec-prelim}, we set up our notation and give a brief overview of Berkovich spaces, global Arakelov--Zhang pairing and quantitative equidistribution theorem. Finally, in Section \ref{sec-proof}, we provide a proof of our main theorem along with some auxiliary results. We note that the main ideas of our work are coherent with \cite{Prout, yap}.

\section{Preliminaries}\label{sec-prelim}
Let $K\subset \mathbb{C}$ be a number field which is Galois over $\mathbb{Q}$ with the Galois group $\mathrm{Gal}(K/\mathbb{Q})$. Let $M_K$ be the set of places of $K$. For each place $w\in M_K$, let $K_w$ denote the completion of the number field $K$ with respect to $w$ and $v$ is the restriction  of $w$ to $\mathbb{Q}$. For every $x\in K$, we define the normalized absolute value $|\cdot |_w$ as follows,
$$|x|_w:=|\mbox{Norm}_{K_w/\mathbb{Q}_v}(x)|_v^{{1/[K:\mathbb{Q}]}}.$$
In this notation, we have the product formula
\[\displaystyle\prod_{\omega\in M_K}|x|_\omega=1,\] for all $x\in K^\times$.
\subsection{Canonical heights} 
Here, we define the global canonical
height function associated to a polynomial $f\in K[z]$ of degree $d \geq 2$ and summarize some of its main properties.

If $v \in M_K$, we define the {\it standard local height function} on
$\C_v$ to be the function $\lambda_v : \C_v \to \R$ given by 
\[
\lambda_v(z) = \log^+ | z |_v,
\]
where for $x \in \R$ and $\log^+ x = \log \max \{ x,1 \}$.
If $n_v=[K_v:\mathbb{\Q}_v]$ is the local degree of $K$, then the usual absolute height of $\alpha\in K$ is given by
\begin{alignat*}{1}
h(\alpha)&=\frac{1}{[K:\mathbb{\Q}]}\sum_{v\in M_K}n_v\lambda_v(\alpha).
\end{alignat*} 
If $U$ is any finite, Galois-stable subset of $\overline{K}$, we define the
{\it absolute logarithmic height} of $U$ to be the average of the (properly normalized) local heights of all elements in $U$.  More precisely, we define
\[
h(U) := \frac{1}{|U|} \sum_{z \in U} \frac{1}{[K:\mathbb{\Q}]}\sum_{v \in M_K} n_v\lambda_v(z).
\]
If $U = \{ \alpha \}$ consists of a single algebraic number,
the height $h(\alpha)$ coincides with the usual
definition of absolute logarithmic height.

We now define the (global) canonical height associated to 
$f \in K[z]$ by the formula
\[
\hhat_f (z) := \lim_{n\to\infty} \frac{1}{d^n} h(f^n(z)),
\]
where $f^n$ denotes the $n$th iterate of $f$.  It follows from \cite{CS} that the limit exists and in particular that
$\hhat_f$ is well-defined.  
As a function from $\overline{K}$ to $\R$, we can characterize $h_f$ as the unique function such that 
\begin{itemize}
\item[(1)] There exists a constant $M$ (depending only on $f$) such that
  for all $\alpha \in \overline{K}$, we have $|\hhat_f(\alpha) - h(\alpha)| \leq M$.
\item[(2)] For all $\alpha \in \overline{K}$, we have $\hhat_f(f(\alpha)) = dh_f(\alpha)$.
\end{itemize}
Since $K$ satisfies the Northcott finiteness property, these two properties imply:
\begin{itemize}
\item[(3)] If $\alpha \in \overline{K}$, then $\hhat_f(\alpha) \geq 0$ and 
$\hhat_f(\alpha) = 0$ if and only if $\alpha$ is a {\it preperiodic point}
for $f$.  
\end{itemize}
There is a decomposition of the global canonical height attached to $f$
into a sum of canonical local  heights. For any place $v\in M_K$ and for $f(z)\in\mathbb{C}_v[z]$ a polynomial of degree $d\geq 2$,
the associated (\emph{Call-Silverman}) \emph{canonical local height function} $\lambda_{f,v} : \C_v \to \R$ is given by
\[
\lambda_{f,v}(z) := \lim_{n \to \infty} \frac{1}{d^n} \lambda_v(f^n(z)).
\]
Call--Silverman heights were introduced in \cite{CS} and \cite{sil}. It is shown in \cite{CS} that this limit exists and that we have
\[
\hhat_f (\alpha) =\frac{1}{[K:\mathbb{Q}]}\sum_{v\in M_{K}}n_v\lambda_{f, v}(\alpha),
\]
where $n_v=[K_v:\mathbb{Q}_v]$. Note that the function $\lambda_{f,v}$ takes non-negative values and it is strictly positive exactly at points $z\in\C_v$ for which $f^n(z)\to\infty$ as $n\to\infty$.
That is, $\lambda_{f,v}$ is zero precisely on the \emph{filled Julia set} of $f$ at $v$. Moreover, $\lambda_{f,v}$ differs from the
standard local height function $\lambda_v(z)$
by a bounded amount and the two coincide for all but finitely many $v$.

If $f(z) = a_0 + a_1z + \cdots + a_d z^d  \in K[z]$ with $a_d \neq
0$ and if $v$ is non-archimedean, we let $\alpha_v$ be the
minimum of the $d+1$ numbers $\frac{1}{d-i}v(\frac{a_i}{a_d}), \;
0\leq i < d$ and $\frac{1}{d-1} v(\frac{1}{a_d})$.  (We consider $v(0)$ to be $+\infty$). We define 
\[
c_v(f) := |a_d|_v^{\frac{-1}{d-1}}.
\]
The following are some of the properties of the canonical local
heights attached to $f$ (see \cite{CS} for proofs):

\begin{enumerate}
\item For each $v$ and each $z \in \C_v$, $\lambda_{f,v}(f(z)) =
d \lambda_{f,v}(z)$.
\item For each $v$, $\lambda_{f,v} : \C_v \to \R$ is a continuous
non-negative function which is zero precisely on the {\it $v$-adic filled Julia set} $\mathcal{K}_v := \{ z \in \C_v \; : \; |f^n(z)|_v \not\to \infty
\}$ of $f$.
\item For each $v$, the difference $|\lambda_{f,v} - \lambda_v|$ is a bounded
function on $\C_v$.
\item For all but finitely many $v \in M_K$, we have
$\lambda_{f,v} = \lambda_v$.
\item If $v$ is archimedean, then
$\lambda_{f,v}(z) = \log |z|_v - \log c_v(f) + o(1)$
as $|z|_v \to \infty$.
\item If $v$ is non-archimedean, then for
$|z|_v$ sufficiently large (depending only on $f$), we have
$\lambda_{f,v}(z) = \log |z|_v - \log c_v(f).$
Specifically, this formula is valid whenever $v(z) < \alpha_v$.
\end{enumerate}
We need the following result to prove our theorem. Here we mention a part of Lemma 3.4 in \cite{Habegger}.

\begin{lemma} [\cite{Habegger}]\label{absolut} 
    Let $K$ be a number field, $f(z):=z^d+c$ be the unicritical family of degree $d\ge 2$ defined over $\overline{K}$. Suppose $\{ f^{n}(0) : n \in \mathbb{N} \}$
is a bounded set. Let $v$ is an archimedean place, we abbreviate $\lambda_{f,v}$ by $\lambda_f$. Then $|c| \le 2^{1/(d-1)}$.
\end{lemma}
\subsection{Bifurcation measure} 
We also let $\PP^1_{\mathrm{Berk},v}$ be the {\em Berkovich projective line} over $\C_v$.    This is a canonically defined path-connected compact Hausdorff space containing $\PP^1(\C_v)$ as a dense subspace. For each $v\in M_k$, we fix an embedding of $\overline{K}$ into $\C_v$. We remark here that if $v$ is archimedean, we have $\C_v \simeq \C$ and $\PP^1_{\mathrm{Berk},v} \simeq \PP^1(\C)$. In particular, for each $a\in \C_v$ and $r>0$, there is a point $\zeta(a,r)\in\PP^1_{\mathrm{Berk},v}$
corresponding to the closed disk $\overline{D}(a,r):=\{x\in\C_v : |x-a|_v\leq r\}$.
The Berkovich point $\zeta(0,1)$ corresponding to the closed unit disk is called
the \emph{Gauss point}. See (\cite[Chapters~1--2]{baker2010} or \cite{berko}) for more information on $\PP^1_{\mathrm{Berk},v}$.

For each $v \in M_k$ there is a distribution-valued Laplacian operator $\Delta$ on $\PP^1_{\mathrm{Berk},v}$. For its definition and some examples we refer the reader to  \cite{baker2010}. An  important example is the Laplacian of $\log^{+}|x|_v$. Note that the function
$\log^+|x|_v$, which is originally defined on $\PP^1(\C_v)\setminus\{\infty\}$, extends naturally to a continuous real valued function defined on $\PP^1_{\mathrm{Berk},v} \backslash \{ \infty \}$. The Laplacian of its extension, also denoted by $\log^{+}|x|_v$, is
\begin{equation}\label{lambdav potential}
\Delta \log^+|x|_v = \delta_{\infty} - \lambda_v,
\end{equation}
where $\delta_{\infty}$ is the dirac measure supported at $\infty$ and $\lambda_v$ is the uniform probability measure on the complex unit circle $\{ |x|_v = 1 \}$ when $v$ is archimedean and a point mass at the Gauss point of $\PP^1_{\mathrm{Berk},v}$ when $v$ is non-archimedean. A probability measure $\mu_v$ on $\PP^1_{\mathrm{Berk},v}$ is said to have {\em continuous potentials} if $\mu_v - \lambda_v = \Delta g$
for some continuous function $g : \mathbb{P}^1_{\mathrm{Berk}, v} \to \mathbb{R}$.  We call the function $g$ a potential of $\mu_v$ and note that any two potentials of $\mu_v$ differ by a constant. 

As in \cite{Benedetto}, for the polynomial $f(z)= z^d+c$, we denote the associated canonical local height function at $v$ by $\lambda_{d,c,v}$ and define the Green's function $G_{d,v}:\C_v\to\mathbb{R}$ by 
\begin{equation} \label{green}
    G_{d,v}(c):=\lambda_{d,c,v}(c) = \lim_{n\to\infty} \frac{1}{d^n} \lambda_v(f_{d,c}^n(c)).
\end{equation}
That is, $G_{d,v}$ measures the $v$-adic escape rate of the critical point of $f_{d,c}$, and hence $G_{d,v}$ is zero precisely on the multibrot set $\mathcal{M}_{d,v}$, and strictly positive on $\C_v \smallsetminus\mathcal{M}_{d,v}$. If $v$ is an archimedean place, so that $\C_v\cong\C$, the (potential-theoretic) Laplacian of $G_{d,v}$ is a probability measure $\mu_{d,v}$ on $\C$,
called the \emph{bifurcation measure} of the family $f_{d,c}$.
The support of $\mu_{d,v}$ is precisely the bifurcation locus $\partial\mathcal{M}_{d,v}$ of the family; see, for example, \cite[Proposition~3.3.(5)]{BD11} or \cite[Sections~4.1--4.2]{GKNY17}. If $v$ is a non-archimedean place, then the sequence $(|f_{d,c}^n(c)|_v)_{n\geq 1}$ is bounded if and only if $|c|_v\leq 1$.
In fact, we have the explicit formula $G_{d,v}(c)=\log\max\{1,|c|_v\}$,
which has a unique continuous extension to $\PP^1_{\mathrm{Berk},v}\smallsetminus\{\infty\}$. The Laplacian of $G_{d,v}$, when restricted to
$\PP^1_{\mathrm{Berk},v}\smallsetminus\{\infty\}$ the desired probability measure is $\mu_{d,v}=\delta_{\zeta(0,1)}$, the delta measure at the Gauss point.

Since the bifurcation measure $\mu_{d,v}$ is defined as the Laplacian of
$G_{d,v}(c)=\lambda_{d,c,v}(c)$, we can
recover the canonical local height $\lambda_{d,c,v}(c)$ by integrating
an appropriate kernel against $\mu_{d,v}$, as follows.

\begin{lemma}[Lemma 2.1, \cite{Benedetto}] \label{lem:locht} 
Let $d\geq 2$ be an integer, let $v\in M_k$ and let $\alpha \in \C_v$.
Let $\lambda_{d,\alpha,v}$ be the canonical local height function for the map $f_{d,\alpha}(z)=z^d+\alpha$.
Let $\mu_{d,v}$ be the bifurcation measure of the family $f_{d,c}$. Then
\begin{equation}
\label{eq:lochtint}
\int_{\PP^1_{\mathrm{Berk},v}} \log |x-\alpha|_v \, d\mu_{d,v}(x)
= \lambda_{d,\alpha,v}(\alpha) .
\end{equation}
\end{lemma}
\subsection{The Arakelov--Zhang pairing}
For two rational maps $\varphi$ and $\psi$ defined on $\mathbb{P}^1$ over a number field $K$, each having degree at least two, the Arakelov--Zhang pairing $\langle \varphi, \psi \rangle$ captures a deep relationship between their respective canonical height functions $h_{\varphi}$ and $h_{\psi}$. 

To define Arakelov--Zhang pairing we recall some notation (see \cite{petsche2012}). Let $K$ be a number field and let $M_K$ denote the set of places of $K$. For each $v \in M_K$, let $K_v$ be the completion of $K$ at $v$ and let $\mathbb{C}_v$ be the completion of an algebraic closure of $K_v$. Let $\mathcal{L}$ be a line bundle on $\mathbb{P}^1$. For each place $v \in M_K$, the line bundle $\mathcal{L}$ extends to a line bundle $\mathcal{L}_v$ on the Berkovich projective line $\mathbb{P}^1_{\mathrm{Berk},v}$. A continuous metric $\|\cdot\|_v$ on $\mathcal{L}_v$ is a continuous function $\|\cdot\|_v : \mathcal{L}_v \to \mathbb{R}_{\ge 0}$ which induces a norm on each fiber $\mathcal{L}_z$ as a $\mathbb{C}_v$-vector space. The metric $\|\cdot\|_v$ is said to be \emph{semi-positive} if for any section $s$, the function $\log |s(z)|$ on $\mathbb{P}^1_{\mathrm{Berk},v}$ is subharmonic. It is said to be \emph{integrable} if $\log |s(z)|:\mathbb{P}^1 \to \mathbb{R} \cup \{-\infty\}$ can be written as the difference of two subharmonic functions.      

For $\mathcal{L} = \mathcal{O}(1)$, there is a standard metric $\|\cdot\|_{\mathrm{st},v}$ defined by
$$\|s(z)\|_{\mathrm{st},v} = \frac{|s(z_1,z_2)|_v}{\max\{|z_1|_v, |z_2|_v\}},$$ where $s(z_1,z_2)$ is the linear homogeneous polynomial in $\mathbb{C}_v[z_1,z_2]$ representing the section $s$ and
$z = (z_1:z_2) \in \mathbb{P}^1(\mathbb{C}_v)$. An integrable adelic metric on $\mathcal{O}(1)$ is a family of metrics $(\|\cdot\|_v)_{v \in M_K}$ such that $\|\cdot\|_v$ is a continuous integrable metric on $\mathcal{O}(1)$ over $\mathbb{P}^1_{\mathrm{Berk},v}$ for each $v$ and $\|\cdot\|_v = \|\cdot\|_{\mathrm{st},v}$ for all but finitely many places $v$.

For a global section $s \in \Gamma(\mathbb{P}^1, \mathcal{O}(1))$ and a point $z \in \mathbb{P}^1(K) \setminus \{\text{div}(s)\}$, the height is given by:
\begin{equation} \label{eqheight}
	h_\mathcal{L}(z) = \sum_{v \in M_K} N_v \log \|s(z)\|_{st,v}^{-1}
\end{equation}
where $N_v = [K_v : \mathbb{Q}_v]/[K : \mathbb{Q}]$.
Let $\varphi: \mathbb{P}^1 \to \mathbb{P}^1$ be a rational function of degree $d \geq 2$ defined over $K$ and let $\epsilon : \mathcal{O}(d) \overset{\sim}{\to} \varphi^* \mathcal{O}(1)$ be a $K$-isomorphism, known as a $K$-polarization. The canonical adelic metric associated to $(\varphi, \epsilon)$ on $\mathcal{O}(1)$ is then the family $\|\cdot\|_{\varphi, \epsilon} = (\|\cdot\|_{\varphi, \epsilon, v})$ defined for all $v \in M_K$. The canonical height relative to $\varphi$ is defined by
\begin{equation} \label{eqheight1}
	h_{\varphi}(z) = \sum_{v \in M_K} N_v \log \|s(z)\|_{\varphi, \epsilon, v}^{-1}
\end{equation}
for all $z \in \mathbb{P}^1(K) \setminus \{\text{div}(s)\}$.
Let $\alpha \in K$ and define the section $s(z) = z_1 - \alpha z_2$ of $\mathcal{O}(1)$. Define $L_{\alpha}$ to be the adelic line bundle where, for each place $v \in M_K$, the metric is given by $\log \|s(z)\|_{st, v}^{-1} = \lambda_{\alpha, v}(z)$. 
Let $\mathcal{G} \subset \mathbb{P}^1(K)$ be a finite $\mathrm{Gal}(\overline{K}/K)$-invariant set and let $s$ be a section of $\mathcal{O}(1)$ such that $\mathrm{div}(s) \notin \mathcal{G}$.
The height of $\mathcal{G}$ with respect to the adelic line bundle $\mathcal{L}_{\alpha}$ is defined by $$h_{\mathcal{L}_{\alpha}}(\mathcal{G}) = \frac{1}{|\mathcal{G}|}
\sum_{z \in \mathcal{G}} \sum_{v \in M_K} N_v \log \|s(z)\|_v^{-1},$$ where $N_v = \frac{[K_v:\mathbb{Q}_v]}{[K:\mathbb{Q}] }$.

For a rational map $\varphi : \mathbb{P}^1 \to \mathbb{P}^1$  defined over $K$
with degree $d \ge 2$, let $L_\varphi$ denote the canonical adelic line bundle associated to $\varphi$ (see Section~3.5 of \cite{petsche2012}). The associated height function $h_{\mathcal{L}_\varphi}$ coincides with the canonical
height $h_\varphi$.

For any two rational maps $\varphi : \mathbb{P}^1 \to \mathbb{P}^1$ and $\psi : \mathbb{P}^1 \to \mathbb{P}^1$, defined over $K$, $\mathcal{L}_\varphi$ and $\mathcal{L}_\psi$ be the adelic line bundle associated to $\varphi$ and $\psi$ respectively. Let $s, t \in \Gamma(\mathbb{P}^1, \mathcal{O}(1))$ be two sections with $\operatorname{div}(s) \neq \operatorname{div}(t)$. We define the local Arakelov--Zhang pairing of $\mathcal{L}_\varphi$ and $\mathcal{L}_\psi$, with respect to the sections $s$ and $t$, by
\begin{align*}
    \langle \mathcal{L}_\varphi, \mathcal{L}_\psi \rangle_{s, t, v} &= - \int \left\{ \log \| s(z) \|_{\varphi,\epsilon_{\varphi}} \right\} \, d\Delta \left\{ \log \| t(z) \|_{\psi,\epsilon_{\psi}} \right\} \\
    &= \log \| s(\operatorname{div}(t)) \|_{\varphi,\epsilon_{\varphi}} - \int \log \| s(z) \|_{\varphi,\epsilon_{\varphi}} \, d\mu_{\psi}(z),
\end{align*}
where $\epsilon_{\varphi}$ and $\epsilon_{\psi}$ are any polarizations of
$\varphi$ and $\psi$, respectively and $\Delta \{-\log \|s(z)\|_{\varphi, \epsilon}\}=\del_{\mathrm{div}(s)}(z)-\mu_\varphi(z)$. Note that $\langle \mathcal{L}_\varphi, \mathcal{L}_\psi \rangle_{s, t, v}$ does not depend on the choice of polarizations $\epsilon_{\varphi}$ and $\epsilon_{\psi}$. The global Arakelov--Zhang pairing is then defined as 
\begin{align*}
\langle \mathcal{L}_\varphi, \mathcal{L}_\psi \rangle &= \sum_{v \in M_K} N_v \, \langle \mathcal{L}_\varphi, \mathcal{L}_\psi \rangle_{s,t,v} + h_{\mathcal{L}_{\varphi}}\!\left( \operatorname{div}(t) \right) + h_{\mathcal{L}_{\psi}}\!\left( \operatorname{div}(s) \right)\\
&=\sum_{v \in M_K} N_v \left( - \int \log \| s(z) \|_{\varphi,\epsilon_{\varphi},v} \, d\mu_{\psi,v}(z) - \log \| t(\operatorname{div}(s)) \|_{\psi,\epsilon_{\psi},v} \right).
\end{align*} 
One of the key results we require related to the Arakelov--Zhang pairing is the following.
\begin{theorem}[Corollary 12, \cite{petsche2012}]\label{lemmapairing}
    Let $(x_n)_{n\geq 0}$ be a sequence of distinct points in $\mathbb{P}^1(\overline{K})$ such that $h_{\mathcal{L}_{\psi}}(x_n) \to 0$ then $h_{\mathcal{L}_{\varphi}}(x_n) \to \langle \mathcal{L}_\varphi, \mathcal{L}_\psi \rangle$.
\end{theorem}
\subsection{Equidistribution of post-critically finite points}
Let $M$ be a metric space.
%  which is assumed to be separable and complete.
A sequence  $ \{ \mu_n \}$  of  measures on $M$ is said to 
{\it converge weakly} to a measure $\mu$
%, denoted by $\mu_n\Rightarrow \mu$, 
if
\[
\lim_{n\to \infty}\,\int_{M} \,f d\mu_n = \int_{M}\, f d\mu
\]
for every bounded, continuous function $f : M \to \R$.
For any finite subset $S \subseteq \C_v$,
we denote by $\delta_S$ the probability
measure $\delta_S := \frac{1}{|S|} \sum_{z \in S} \delta_z$, where
$\delta_z$ is the Dirac probability measure supported at the single point
$z \in \C_v$.
Finally, we say a sequence $ \{ S_n \}$ of finite subsets of $\C_v$ is
{\it equidistributed} with respect to the measure $\mu$ if the sequence
$\{ \delta_n \} := \{ \delta_{S_n} \}$ of probability measures converges weakly to $\mu$. 
Now we recall the equidistribution theorem for post-critically finite parameters (see \cite[Theorem 3.1]{GKNY17}).
\begin{theorem}[\cite{GKNY17}] \label{equistribution}
Let $f : X' \times \mathbb{C} \to \mathbb{C}$ be a non-isotrivial, one-dimensional algebraic family of degree $d \ge 2$ unicritical polynomials over a number field $K$, where $X'$ is a Zariski dense, open subset of an irreducible, smooth curve $X$ defined over $\mathbb{C}$. The set of parameters $t \in X'(\overline{K})$,  for which $f(t,z) : \mathbb{C} \to \mathbb{C}$ is post-critically finite equidistributes on the parameter space $X(\mathbb{C})$ (with respect to the normalized bifurcation measure).
\end{theorem}
The equidistribution result states that parameters $c$ for which the critical point $0$ has finite forward orbit under $f_{d,c}$ are equidistributed with respect to the bifurcation measure, was first shown by Levin \cite{levin} (in the classical sense of equidistribution) and it was shown in the stronger (arithmetic) form by Baker and Hsia \cite[Theorem 8.15]{BakerHsia}. The following result is due to Levin \cite{levin} for one-parameter family of unicritical polynomials of degree $d$ and for $d=2$ it was explicitly stated in \cite[Theorem 8.15]{BakerHsia}.
\begin{lemma}[\cite{levin}]\label{eqdistlem}
Let $\mathcal{M}_{d, v}$ be the set of complex numbers $c$ for which $f_{d, c}(z) = z^d + c$ has connected Julia set, that is, the generalized Mandelbrot set. Then
$$
\lim_{n \to \infty} \frac{1}{d^n}
\sum_{f_{d,c}^n(0)=0} \delta_c = \mu,
$$
where $\mu$ is the harmonic measure on $\mathcal{M}_{d, v}$.   
\end{lemma} 

We can now state the quantitative equidistribution result, which yields a precise estimate on the speed of convergence. For a detailed proof, refer to (\cite[Theorem 5]{favre}, \cite[Theorem 3]{Dujardin}).
\begin{proposition}\label{quantequi}
With notation as in Lemma \ref{eqdistlem}, let $\mathcal{G}_n \subset \mathbb{C}_v$ be a sequence of disjoint finite sets, invariant under the absolute Galois group of $\mathbb{Q}$ and contained in the union
$\bigcup_{n \ge k} \{ c \in \mathbb{C}_v: f_{d,c}^n(0) = f_{d,c}^k(0) \}$.
Then there exists a constant $C >0$ and for any compactly supported $\mathcal{C}^1$ function $\varphi$, we have
	\begin{equation}\label{equid}
		\left|
\frac{1}{|\mathcal{G}_n|} \sum_{c \in \mathcal{G}_n} \varphi(c) - \int \varphi \, d\mu \right|_v \le C \left( \frac{\log |\mathcal{G}_n|}{|\mathcal{G}_n|} \right)^{\frac{1}{2}}
\sup \{ |\varphi|_v, |\varphi'|_v \},
	\end{equation}
where $|\mathcal{G}_n|$ denotes the cardinality of $\mathcal{G}_n$ and $\varphi'$ is the differentiation of $\varphi$.  
\end{proposition}

Next we derive a similar bound to that in \eqref{equid} for $\varphi(x)=\lambda_{\tau, v}(x)$, where $0 < \tau < 1$ and for $\alpha \in \overline{K}$ let $\lambda_{\tau, v}(x):=\log^+|x|_v+\log^+|\alpha|_v-\log \max\{\tau, |x-\alpha|_v\}$.
\begin{proposition} \label{quant}
Let $K$ be a number field, $f_{d,c}(z):=z^d+c$ be the unicritical family of degree $d\ge 2$ defined over $\overline{K}$ and let $v$ be a place of $M_K$. Let $\mu_{d,v}$ be a bifurcation measure of the family $f_{d,c}$. Then for all finite $\emph{Gal}(\overline{K}/K)$-invariant set $\mathcal{G}$ and for $\alpha \in \overline{K}$, there exists a constant $C_2 >0$ such that 
\begin{equation} \label{equant}
    \left| \frac{1}{|\mathcal{G}|} \sum_{x \in \mathcal{G}} \lambda_{\tau,v}(x) - \int \lambda_{\tau,v}(x) d\mu_{d,v} \right|_v \leq C_2  \left( \frac{\log|\mathcal{G}|}{|\mathcal{G}|} \right)^{1/2}\left(\log ^+|\alpha|_v+\dfrac{1}{\tau}\right).
\end{equation}    
\end{proposition}
\begin{proof}
    Let $x \in \mathcal{G}$ be a post-critically finite parameter, so by definition \eqref{Mandelbrot}, $x \in \mathcal{M}_{d, v}$. As $\mathcal{M}_{d, v}$ is compact, there exists $R >0$ such that $\mathcal{M}_{d, v} \subset \{x: |x|_v<R\}$. Since $|x-\alpha|_v \leq |x|_v+ |\alpha|_v \leq R+ |\alpha|_v$, it follows that $|\log^+|x|_v|_v \leq \log^+R$ and $$|\log \max\{\tau, |x-\alpha|_v\}|_v \leq  \max \{\log (R+|\alpha|_v), |\log \tau|_v\}.$$ So,
    \begin{equation*}
        |\lambda_{\tau,v}(x)|_v \leq \log ^+R+\log ^+|\alpha|_v+ \max \{\log (R+|\alpha|_v), |\log \tau|_v\}.
    \end{equation*}
    Consequently, there exists a constant $C_3>0 $, depending only on $R$, such that $$\sup_{x \in \mathcal{G}} |\lambda_{\tau,v}(x)|_v \leq C_3(1+\log^+|\alpha|_v+|\log \tau|_v).$$
    On the other hand, we have the derivative estimate $$\sup_{x \in \mathcal{G}} |\lambda_{\tau,v}'(x)|_v =\sup_{x \in \mathcal{G}} \left|\frac{d}{dx}\lambda_{\tau,v}(x)\right|_v <\dfrac{1}{\tau}+1.$$
    Therefore,
\begin{align*}
\sup_{x \in \mathcal{G}} \{ |\lambda_{\tau,v}(x)|_v, |\lambda_{\tau,v}'(x)|_v \}
&= \max \{ \sup_{x \in \mathcal{G}} |\lambda_{\tau,v}(x)|_v, \sup_{x \in \mathcal{G}} |\lambda_{\tau,v}'(x)|_v \} \\
&\leq C_3 \left( 1 + \log^+ |\alpha|_v + |\log \tau|_v \right) 
+ 1 + \frac{1}{\tau}.
\end{align*}
Since $0 < \tau < 1$, we have $|\log \tau|_v \leq \frac{1}{\tau}$. 
Thus, after enlarging the constant if necessary, we obtain
\begin{equation*}
\sup_{x \in \mathcal{G}} \{ |\lambda_{\tau,v}(x)|_v, |\lambda'_{\tau,v}(x)|_v \}
\leq C_4 \left(  \log^+ |\alpha|_v + \frac{1}{\tau} \right).
\end{equation*}
Then, by Proposition \ref{quantequi}, we will obtain \eqref{equant}.
\end{proof}
\begin{proposition}\label{propadelic}
Let $K$ be a number field, $f_{d,c}(z) := z^d+c$ be a post-critically finite map of degree $d \geq 2$ defined on $\mathbb{P}^1(\overline{K})$ and let $\mathcal{L}_{d,c}$ be a canonical adelic line bundle for $f_{d, c}$. Let $(x_n)_{n\geq0}$ be a sequence of distinct post-critically finite parameters in $\mathcal{M}_{d,v}, \alpha \in \overline{K}$ and for each $n$ let $\mathcal{G}_n$ be the set of $\mathrm{Gal}(\overline{K}/K)$ conjugates of $x_n$. Then for any adelic line bundle $\mathcal{L}$ there exists a constant $C_{5}>0$ such that
	$$\left| h_{\mathcal{L}}(\mathcal{G}_n) - \langle \mathcal{L}, \mathcal{L}_{d,c} \rangle \right|_v \leq C_{5}\left(\frac{\log |\mathcal{G}_n|}{|\mathcal{G}_n|}\right)^{1/2}\left(  \log^+ |\alpha|_v + \frac{1}{\tau} \right).$$
    \end{proposition} 
    \begin{proof}
Suppose that $(x_n)_{n \geq 0}$ is a sequence of post-critically finite parameters in $\mathcal{M}_{d,v}$ and $\mathcal{G}_n$ is the $\mathrm{Gal}(\overline{K}/K)$-orbit of $x_n$. The height $(x_n)_{n \geq 0}$ associated to the adelic line bundle $\mathcal{L}$ is given by
\begin{equation}\label{prop2.3eq2}
h_{\mathcal{L}}(x_n) = \frac{1}{|\mathcal{G}_n|}\sum_{x\in \mathcal{G}_n} h_{\mathcal{L}}(x)=\frac{1}{|\mathcal{G}_n|}\sum_{x\in \mathcal{G}_n} 
\sum_{v \in M_K} N_v \log \| s(x) \|_{\mathrm{st},v}^{-1},
\end{equation} 
 where the section $s(x)= x-\alpha$. Since $$h_{d,c}(x_n) = \frac{1}{|\mathcal{G}_n|}\sum_{x\in \mathcal{G}_n}h_{d,c}(x) =\frac{1}{|\mathcal{G}_n|}\sum_{x\in \mathcal{G}_n}\lim_{m \to \infty} \dfrac{h(f_{d,c}^m(x))}{d^m}$$ and for each place $v \in \mathcal{M}_K$, the local Green's function 
$G_{d,v}(x_n)$ vanishes if and only if $x_n$ belongs to the 
$v$-adic Mandelbrot set $\mathcal{M}_{d,v}$, so by \eqref{green}, it follows that $h_{d,c}(x_n) \to 0$. Then by Theorem \ref{lemmapairing}, we have
\begin{equation}\label{prop2.3eq1}
\langle \mathcal{L}, \mathcal{L}_{d,c} \rangle = \lim_{n \to \infty} h_{L}(x_n).
\end{equation} 
For each place $v \in M_K$, we have the local height function 
$\lambda_{\alpha,v}(x) = \log \| s(x) \|_{\mathrm{st},v}^{-1}$.
Then from \eqref{prop2.3eq2} and \eqref{prop2.3eq1},
\begin{equation*}
\langle \mathcal{L}, \mathcal{L}_{d,c} \rangle = \lim_{n \to \infty} \frac{1}{|\mathcal{G}_n|} \sum_{x \in \mathcal{G}_n} \sum_{v \in M_K} N_v \lambda_{\alpha,v}(x).
\end{equation*} 
Let $S$ be a finite set of places including all of the archimedean places. For any place $v \notin S$, the standard metric satisfies $\|s(x)\|_{st,v}^{-1}=1$. Since $x \in \mathcal{M}_{d,v}$ is a post-critically finite parameter, it follows that $|x|_v \leq 1$ and using $S$-integrality, we have the local height function $\lambda_{\alpha,v}(x)$ vanishes identically for all places $v \notin S$,
the sum over $M_K$ reduces to a sum over the finite set $S$, yielding
\begin{equation*}
\langle \mathcal{L}, \mathcal{L}_{d,c} \rangle=\lim_{n \to \infty} \frac{1}{|\mathcal{G}_n|} \sum_{x \in \mathcal{G}_n}
\sum_{v \in S} N_v \lambda_{\alpha,v}(x).
\end{equation*} 
Because the set $S$ is finite and independent of $n$, we can interchange the limit with the sum over $v\in S$. By Theorem \ref{equistribution}, the Galois orbits of $(x_n)_{n \geq 1}$ are equidistributed with respect to the bifurcation measure $\mu_{d,v}$, so
\begin{equation}\label{eqprop2.3}
\langle \mathcal{L}, \mathcal{L}_{d,c} \rangle = \sum_{v \in S} N_v \int \lambda_{\alpha,v}(x) \, d\mu_{d,v}.
\end{equation}
For any $\mathrm{Gal}(\overline{K}/K)$-invariant set $\mathcal{G}_n$, we have $h_{\mathcal{L}}(\mathcal{G}_n)= h_{\mathcal{L}}(x_n)$. Then from \eqref{prop2.3eq2} and \eqref{eqprop2.3}, we get
    \begin{align*}
       \left| h_{\mathcal{L}}(\mathcal{G}_n) - \langle \mathcal{L}, \mathcal{L}_{d,c} \rangle \right|_v &= \sum_{v \in M_K} N_v \left| \frac{1}{|\mathcal{G}_n|} \sum_{x \in \mathcal{G}_n} \lambda_{\alpha,v}(x) - \int \lambda_{\alpha,v}(x) \, d\mu_{d,v} \right|_v \\
      & \leq \sum_{v \in M_K} N_v \left| \frac{1}{|\mathcal{G}_n|} \sum_{x \in \mathcal{G}_n} \lambda_{\tau,v}(x) - \int \lambda_{\tau,v}(x) \, d\mu_{d,v} \right|_v. 
    \end{align*}
By Proposition \ref{quant}, we get
    \begin{align}\label{eq4.2}
        \left| h_{\mathcal{L}}(\mathcal{G}_n) - \langle \mathcal{L}, \mathcal{L}_{d,c} \rangle \right|_v
        & \leq \sum_{v \in M_K} N_v C_{2}  \left( \frac{\log|\mathcal{G}|}{|\mathcal{G}|} \right)^{1/2}\left(\log ^+|\alpha|_v+\dfrac{1}{\tau}\right) \nonumber\\ 
        & \leq C_{5}\left(\frac{\log |\mathcal{G}_n|}{|\mathcal{G}_n|}\right)^{1/2}\left(  \log^+ |\alpha|_v + \frac{1}{\tau} \right),
    \end{align}
  for some suitable constant $C_{5} > 0$.  This completes the proof of Proposition \ref{propadelic}.
\end{proof}	
\subsection{Linear forms in logarithms}
We also require the theory of linear forms in logarithms, initially developed by Baker \cite{alan}.  This theory provides lower bounds for expressions of the form
\begin{equation*}
	\left| a_1^{b_1}\cdots a_n^{b_n}-1\right|
\end{equation*}
in terms of the heights of $a_i$ and $b_i$. The following version is given by B\'erczes,  Evertse and Gy\"ory \cite{Berczes}.
\begin{theorem}\label{linearthm}
Let $\alpha_1,\ldots,\alpha_n$ be $n\geq 2$ nonzero elements of $K$ and $b_1,\ldots,b_n \in \mathbb{Z}$, not all zero. Define
\begin{equation*}
\Lambda := \alpha_1^{b_1}\cdots \alpha_n^{b_n} - 1,
\end{equation*}
\begin{equation*}
\Theta := \prod_{i=1}^n \max\left\{ h(\alpha_i), \frac{2}{[K:\mathbb{Q}] \bigl(\log(3[K:\mathbb{Q}])\bigr)^3 }\right\},
\end{equation*}
and
\begin{equation*}
B := \max\left\{3, |b_1|,\ldots,|b_n| \right\}.
\end{equation*}
Let $v$ be a place of $K$ and write
\begin{equation*}
N(v) :=
\begin{cases}
2, & \text{if $v$ is infinite}, \\
N_{K/\mathbb{Q}}(\mathfrak{p}), & \text{if $v=\mathfrak{p}$ is finite.}
\end{cases}
\end{equation*}
Suppose that $\Lambda \neq 0$. Then for $v \in M_K$, we have
\begin{equation*}
\log |\Lambda|_v > -c_1(n,[K:\mathbb{Q}]) \dfrac{N(v)}{\log N(v)}  \Theta \log B,
\end{equation*}
where
\begin{equation*}
c_1(n,d) = 12d(16ed)^{3n+2}\max(1,\log d)^2,
\end{equation*}
with an extra factor of $d$ arising from the chosen normalization of the absolute values.
\end{theorem}

\section{Proof of Main Results}\label{sec-proof}
The following proposition provides a bound on the $v$-adic logarithmic distance between a non post-critically finite parameter $\alpha$ and a post-critically finite parameter $x$ that can be determined in terms of the height of $\alpha$, the degree of $f_{d, c}$ and the size of the Galois orbit $\mathcal{G}$. 
\begin{proposition}\label{prop3.1}
	Let $K$ be a number field and $f_{d,c}(z):=z^d+c$ be the unicritical family of degree $d\ge 2$ defined over $\overline{K}$. Let $\mathcal{G}$ be any $\emph{Gal}(\overline{K}/K)$-orbit of some post-critically finite parameter, let $v$ be an archimedean place of $K$ and $\alpha \not \in \partial \mathcal{M}_{d, v} $ be a non post-critically finite parameter. Then for any $\epsilon > 0$, there exists a constant $C_6>0$ such that  
	$$ \max_{x \in \mathcal{G}} \log |x-\alpha|_v^{-1} \leq C_6  \left(h(\alpha)+ \frac{d}{d-1}\right)|\mathcal{G}|^{8+\epsilon}.$$ 
\end{proposition}
\begin{proof}
Since $v$ is an archimedean place and $\alpha \not \in \partial \mathcal{M}_{d, v}$, there exists a real number  $r>0$ such that the open disk
$D(\alpha,r):=\{y\in \mathbb{C}_v : |y-\alpha|_v<r\}$
does not intersect the closed set $\partial \mathcal{M}_{d,v}$. 
If $\alpha \notin \mathcal{M}_{d,v}$, then the map $f_{d,\gamma}$ cannot be post-critically finite for any 
$\gamma \in D(\alpha,r)$. On the other hand, if $\alpha \in \mathcal{M}_{d,v}$, then there exists a neighborhood of $\alpha$ that contains no post-critically finite parameter. Consequently, after possibly replacing $r$ by a smaller positive number, we may assume that $f_{d,\gamma}$ is not post-critically finite for any $\gamma \in D(\alpha,r)$. Thus, for any post-critically finite parameter $x$, we have $|x-\alpha|_v>r$ for some $r>0$.

For every post-critically finite parameter $x \in \mathcal{G}$, we have $|\alpha|_v - |x|_v \le |x-\alpha|_v.$
First, suppose that $|x|_v \neq |\alpha|_v$. Then
\begin{equation}\label{peq1}
\log \left| |\alpha|_v - |x|_v \right|_v \le \log |x-\alpha|_v,
\end{equation}
and hence by Lemma \ref{absolut}, we get
\begin{align*}
\log |x-\alpha|_v^{-1} &\leq \log \left| |\alpha|_v - |x|_v \right|_v^{-1} \leq h( |\alpha|_v - |x|_v)
\\& \leq h(|\alpha|_v)+h(|x|_v)+\log 2 \\&\le h(|\alpha|_v)+h\bigl(2^{1/(d-1)}\bigr)+\log 2.
\end{align*} 
Since
\begin{align*}
h\bigl(2^{1/(d-1)}\bigr) &= \log\max\left\{ \left|2\right|_v^{1/(d-1)},1\right\} = \frac{1}{d-1}\log 2
\end{align*} 
and $h(|\alpha|_v)\le 2h(\alpha)$, then
\begin{align*}
\log |x-\alpha|_v^{-1}
&\le 2h(\alpha)+\frac{1}{d-1}\log 2+\log 2 = 2h(\alpha)+\left(\frac{d}{d-1}\right)\log 2.
\end{align*} 
Thus,
\begin{align}\label{peq2}
\log |x-\alpha|_v^{-1}
&\le [\Q(\alpha):\Q]\left(h(\alpha)+\frac{d}{d-1}\right).
\end{align}
Next, if $|x|_v = |\alpha|_v$, then $\log \left| |\alpha|_v - |x|_v \right|_v$ is not a continuous function. So, to find the upper bound for $\log |x-\alpha|_v^{-1}$, we will employ linear forms in logarithms. Observe that $\log |x-\alpha|_v= \log |\alpha|_v+ \log |x \alpha^{-1}-1|_v$. We will apply Theorem \ref{linearthm} with $\alpha_1=x, \alpha_2=\alpha, b_1=1, b_2=-1$ and set $D'=[\Q(\alpha, x):\Q]$.
\begin{align} \label{eqlinear}
    \begin{split}
        \log |x \alpha^{-1}-1|_v &>-C(2, D')\frac{2\log 3}{\log 2} \max \left(h(x), \frac{2}{D'(\log(3D')^3)}\right)\max \left(h(\alpha), \frac{2}{D'(\log(3D'))^3}\right)\\
        & \geq -12 D' (16eD')^8 (\log D')^2 \frac{2}{D'(\log(3D'))^3} h(\alpha) \log 3\\
        & \geq -24 \log 3 \cdot (16e)^8 D'^8 \frac{1}{\log D'} h(\alpha).
    \end{split}
\end{align}
 Using \cite[Proposition 1.21]{morandi}, we can write $[\Q(\alpha, x):\Q] \leq [\Q(\alpha):\Q][\Q(x):\Q]$. As $|\mathcal{G}|=[\Q(x):\Q]$, we can write \eqref{eqlinear} as 
\begin{align}
    \begin{split}
         \log |x \alpha^{-1}-1|_v^{-1} &\leq 24 (16e)^8 [\Q(\alpha):\Q]^8 \frac{1}{\log D'} |\mathcal{G}|^8 h(\alpha) \\
         & \leq 24 (16e)^8 [\Q(\alpha):\Q]^{8+\epsilon} |\mathcal{G}|^{8+\epsilon} h(\alpha)
    \end{split}
\end{align}
The last inequality is true because of $\frac{1}{\log D'} \leq D'^\epsilon$ for some $\epsilon >0$.
So, 
\begin{align}
    \begin{split}
        \log |x-\alpha|_v^{-1} &=\log |\alpha|_v^{-1}+\log |x\alpha^{-1}-1|_v^{-1}\\
        & \leq \log |\alpha|_v^{-1}+ 24 (16e)^8 [\Q(\alpha):\Q]^{8+\epsilon} |\mathcal{G}|^{8+\epsilon} h(\alpha) \\
        & \leq 24 (16e)^8 [\Q(\alpha):\Q]^{8+\epsilon} |\mathcal{G}|^{8+\epsilon} (h(\alpha)+1).
    \end{split}
\end{align}
Hence, for a suitable constant $C_6 >0$, we have
\begin{equation*}
    \max_{x \in \mathcal{G}} \log |x-\alpha|_v^{-1} \leq C_6  \left(h(\alpha)+ \frac{d}{d-1}\right)|\mathcal{G}|^{8+\epsilon}.
\end{equation*}
\end{proof}
\begin{proposition}\label{archprop}
    Let $f_{d,c}(z):=z^d+c$ be the unicritical family of degree $d\ge 2$ defined over $\overline{K}$. Let $\mathcal{G}$ be any $\emph{Gal}(\overline{K}/K)$-orbit of some post-critically finite parameter, let $v$ is an archimedean place of $K$ and for any $\alpha \in \mathbb{P}^{1}(\overline{K})$, there do not exist two distinct post-critically finite parameters $x_i \in \mathbb{\overline{Q}},i=1,2$,  for $f_{d,c}$ such that $$\log|x_i-\alpha|_v^{-1} > \log M$$ for some $M>1$.
\end{proposition}
\begin{proof}
    Let $v$ be an archimedean place and suppose 
\begin{equation}\label{Mbound}
  |\mathcal{G}|\geq D \quad \mbox{and}\;\;  M \geq 2(2|\mathcal{G}|^2)^{|\mathcal{G}|^2}H^{|\mathcal{G}|^2-1},
\end{equation}
 where $D=[K:\Q]$ and $H$ denotes the height of the polynomial whose roots are the post-critically finite parameters. We claim that there is at most one $x$ inside $\mathcal{G}$ for which $\log \left| x - \alpha \right|_v^{-1} \geq \log M$.  
Consider the disk, $B(\alpha, M^{-1})=\{x\in \mathcal{G}:|x-\alpha|_v \leq M^{-1}\}$. If $x_1$ and $x_2$
are two distinct elements in the disk $B(\alpha, M^{-1})$, then 
\begin{equation}\label{prop4.5eq1}
    |x_1-x_2|_v \leq |x_1-\alpha|_v + |x_2-\alpha|_v\leq 2 M^{-1}.
\end{equation}
Next we will find the lower bound for the difference of two post-critically finite parameters $x_1$ and $x_2$. Using the bound in (\cite[Theorem B]{YM2004}, \cite[Theorem 2]{mahler}), we have
\begin{equation}\label{lbound}
    |x_1-x_2|_v \geq \frac{\sqrt{3}}{(d_c+1)^{\frac{2d_c+1}{2}}}H^{1-d_c} \geq(2d_c)^{-d_c}H^{1-d_c} \geq (2|\mathcal{G}|^2)^{-|\mathcal{G}|^2}H^{1-|\mathcal{G}|^2}, 
\end{equation}
where $d_c$ is the degree of the polynomial whose roots are the post-critically finite parameters. Then using \eqref{prop4.5eq1} and \eqref{lbound}, we get, $$M < (2|\mathcal{G}|^2)^{|\mathcal{G}|^2}H^{|\mathcal{G}|^2-1},$$
which is a contradiction to \eqref{Mbound}. Hence, there are at most one $x$ inside $\mathcal{G}$ for which $\log |x-\alpha|_v^{-1} > \log M$.
\end{proof}
\begin{proposition}\label{prop3.2}
	Let $f_{d,c}(z):=z^d+c$ be the unicritical family of degree $d\ge 2$ defined over $\overline{K}$. Let $\mathcal{G}$ be any $\emph{Gal}(\overline{K}/K)$-orbit of some post-critically finite parameter $x$. Fix a non-archimedean place $v$ of $K$ corresponding to the prime $p$. Then there exists a constant $C_7 > 0$ such that for any $\alpha \in \mathbb{P}^1(\overline{K})$,  we have 
	$$ \max_{x \in \mathcal{G}} \log|x - \alpha|_v^{-1} < \log\left(\dfrac{1}{r}\right)$$
	if $|\mathcal{G}| > C_7$. 
\end{proposition}
\begin{proof}
Suppose that $|\mathcal{G}|>C_7$. We know that each post-critically finite parameter $ \sigma(x_n)$ for the family $f_{d,c}$ is an algebraic integer and therefore $|\sigma(x_n)|_v\le 1$ for every $v$ and for every $K$-embedding $\sigma$. 
Thus, if $|\alpha|_v>1$, then
\begin{align*}
|\sigma(x_n)-\alpha|_v
&= \max\{|\sigma(x_n)|_v,|\alpha|_v\} = |\alpha|_v .
\end{align*} 
Then 
\begin{equation*}
 \log |\sigma(x_n) -\alpha|_v^{-1} = \log \frac{1}{|\alpha|_v},
\end{equation*}
and we are done in this case. Next if $|\alpha|_v\le 1$, then by  \cite[Theorem~1.4]{Benedetto}, for every $0<r<1$ there are only finitely many post-critically finite parameters of the family $f _{d,c}$ in the disk  $\overline{D}(\alpha,r)$. Therefore for any $0<r<1$, there is some $N \geq 1$ such that $r < |\sigma(x_n) -\alpha|_v <1$ for any $n \geq N$ and for any $K$-embedding $\sigma$. 
\begin{comment}
    Let $x_1,\dots,x_n$ be those post-critically finite parameters inside $B(\alpha,r)$.
Let
\begin{align*}
r_0=\min_{1\le i\le n} |x_i-\alpha|_v.\\
|x-\alpha|_v \ge &r_0\\
\log |x-\alpha|_v &\ge \log r_0 ,\\
\log |x-\alpha|_v^{-1} &\le \log \left(\frac{1}{r_0}\right).
\end{align*}
\end{comment}
Then 
\begin{equation*}
    \log |\sigma(x_n) -\alpha|_v^{-1} \leq  \log \frac{1}{r}.
\end{equation*}
\end{proof}
From the proof of the above result, we have the following:
\begin{coro}\label{cor3.3}
	Let $f_{d,c}(z):=z^d+c$ be the unicritical family of degree $d\ge 2$ defined over $\overline{K}$. Fix a non-archimedean place $v$ of $K$ corresponding to the prime $p$. Then for any $\alpha \in \mathbb{P}^{1}(\overline{K})$, there do not exist two distinct post-critically finite parameters $x_i \in \mathbb{\overline{Q}},i=1,2$,  for $f_{d,c}$ and there exist $0 < r_0 <r $ such that $$\log|x_i-\alpha|_v^{-1} > \log \dfrac{1}{r_0}.$$
\end{coro}
\begin{proof}
There are at most $N$ numbers of post-critically finite parameters inside $\overline{D}(\alpha,r)$ satisfying $|x_i-\alpha|_v <r$. So
\begin{equation*}
    \log |x_i-\alpha|_v <\log r,
\end{equation*}
 which gives
    $$\log|x_i-\alpha|_v^{-1} > \log \dfrac{1}{r}.$$
Let $r_0:=\min_{1\le i\le N} |x_i-\alpha|_v$. If $\log |x_i-\alpha|^{-1}_v > \log \left(\dfrac{1}{r_0}\right)$ for $i=1, 2$, then $|x_i-\alpha|_v < r_0$, which is not possible, as $r_0$ is the minimum distance.
\end{proof}
The following result is not required to prove our main theorem. This result is of independent interest, which relates the degree of unicritical polynomial and the degree of number field attached with post-critically finite parameter. 
\begin{proposition}\label{lowerbond}
Let $f_{d,c}(z):=z^d+c$ be the unicritical family of degree $d\ge 2$ defined over $\overline{K}$ and $x$ be a post-critically finite parameter for $f_{d,c}$. Then there exists an integer $n\geq 1$ such that the algebraic degree of $x$ satisfies
\begin{equation*}
[K(x):K] \ge \frac{d^{n-1}}{[K:\mathbb{Q}]}. 
\end{equation*}
\end{proposition}
\begin{proof}
Let $d \ge 2$ be an integer and consider the unicritical family $f_{d,c}(z)=z^d+c$.  We will consider two separate cases. In the first case, let $x \in \overline{\mathbb{Q}}$ be a parameter such that the critical point $0$ is periodic of exact period $n$ under $f_{d,c}$ and root of the dynatomic polynomial $\Phi_n$.  
Then 
\begin{equation*}
\deg \Phi_n \geq \sum_{k \mid n} \mu \left(\frac{n}{k}\right) d^{k-1} ,
\end{equation*}
where $\mu$ denotes the M\"obius function. So,
\begin{align*}
    \deg \Phi_n &\geq \sum_{k \mid n} \mu \left(\frac{n}{k}\right) d^{k-1}\geq d^n- \sum_{\substack{k\mid n \\ k<n}}|d^{k-1}|\geq d^n-\tau(n)d^{n/2}\\
    & \geq d^n\left(1-\dfrac{\tau(n)}{d^{n/2}}\right) \geq \dfrac{d^n}{2},
\end{align*}
where $\tau(n)$ is the number of positive divisors of $n$. Hence,
\begin{equation*}
[\mathbb{Q}(x):\mathbb{Q}] = \deg \Phi_n \geq \dfrac{d^n}{2}\geq d^{n-1},
\end{equation*}
since $d\geq 2$. In the second case, let the critical point $0$ is strictly preperiodic. Define $g_k(x)= f_{d,x}^k(0)$ and since $0$ is strictly preperiodic then for $n>m$, define $P_{m,n}(x):=g_n(x)-g_m(x)=0$. The distinct roots of $P_{m,n}$ are the post-critically parameters. By induction, we have $\deg g_k= d^{k-1}$. Hence $\deg P_{m,n}=\deg g_n= d^{n-1}$. 
Therefore,
\begin{equation*}
[K(x):K]=\frac{[K(x):\mathbb{Q}]}{[K:\mathbb{Q}]} \geq \frac{[\mathbb{Q}(x):\mathbb{Q}]}{[K:\mathbb{Q}]} \ge \frac{d^{n-1}}{[K:\mathbb{Q}]},
\end{equation*}
which completes the proof.
\end{proof}

\begin{proposition}\label{Propalpha}
    Let $K$ be a number field and $f_{d,c}(z):=z^d+c$ be the unicritical family of degree $d\ge 2$ defined over $\overline{K}$. Let $\mathcal{M}_{d,v} \subset \C_v$ denote the Mandelbrot set, let $\mu_{d,v}$ be the harmonic (bifurcation) measure on $\mathcal{M}_{d,v}$. Let $\mathcal{G} \subset \{c \in \C_v: f_{d,c}^n(0)=f_{d,c}^m(0)\}$ be any $\emph{Gal}(\overline{K}/K)$-orbit of some post-critically finite parameter. For $\alpha  \in \overline{K}$, suppose that $f_{d,\alpha}$ is not post-critically finite and for every archimedean place $v$ of $K, \;\alpha \notin \partial \mathcal{M}_{d,v}$. Then there exists a constant $C_8>0$ such that 
	\begin{equation*}
		\left| \frac{1}{|\mathcal{G}|} \sum_{x \in \mathcal{G}} \lambda_{\alpha, v}(x) - \int_{\mathcal{M}_{d,v}} \lambda_{\alpha, v}(x) d\mu_{d,v} \right|_v \leq \frac{C_8}{|\mathcal{G}|}+O\left(\dfrac{\sqrt{\log |\mathcal{G}|}}{\sqrt{|\mathcal{G}|}}\right)\left(\log ^+|\alpha|_v+\dfrac{1}{\tau}\right).
	\end{equation*}
\end{proposition}
\begin{proof}
Let $ [K:\mathbb{Q}]=D$ and $M=1/\tau>1$ be a real number. We split $\mathcal{G}$ into two disjoint sets $\mathcal{G}_1$ and $\mathcal{G}_2$, where $\mathcal{G}_1=\{x\in \mathcal{G}:\log |x-\alpha|_v^{-1} \leq \log M\}$ and $\mathcal{G}_2=\mathcal{G} \setminus \mathcal{G}_1$. For archimedean $v$, we have $|\mathcal{G}_2| \leq 1$ by Proposition \ref{archprop} and for non-archimedean $v$,  we also have $|\mathcal{G}_2| \leq 1$ by Corollary \ref{cor3.3}. 

Now observe that $\lambda_{\tau, v}(x)$ can be written as $$\lambda_{\tau, v}(x) = \lambda_{M, v}(x) = \log^+|x|_v+\log^+|\alpha|_v+  \min(\log M, -\log |x-\alpha|_v).$$
Then for all $x \in \mathcal{G}_1$, we have $\lambda_{M, v}(x)=\lambda_{\alpha, v}(x)$ and hence by Proposition \ref{prop3.1}
\begin{align}\label{E1}
    \left |\frac{1}{|\mathcal{G}|} \sum_{x \in \mathcal{G}} \lambda_{M, v}( x) - \frac{1}{|\mathcal{G}|}\sum_{x \in \mathcal{G}} \lambda_{\alpha, v}(x)  \right |_v&=\left |\frac{1}{|\mathcal{G}|} \sum_{x \in \mathcal{G}_2} \lambda_{M, v}(x) - \frac{1}{|\mathcal{G}|}\sum_{x \in \mathcal{G}_2} \lambda_{\alpha, v}(x)  \right |_v \nonumber \\ &\leq \frac{1}{|\mathcal{G}|} \sum_{x \in \mathcal{G}_2} \left| \lambda_{M, v}(x) -\lambda_{\alpha, v}(x)\right|_v \nonumber\\ & \leq \frac{|\mathcal{G}_2|}{|\mathcal{G}|}  2\log |x-\alpha|^{-1}_v \leq \frac{|\mathcal{G}_2|}{|\mathcal{G}|}C_8,
\end{align} 
where $C_8=2\max_{x \in \mathcal{G}}\log |x-\alpha|^{-1}_v$. As $\mu_{d, v}(B(x, \epsilon))=O(\epsilon)$, then
\begin{align}\label{E2}
    \left|\int \lambda_{M, v}(x)  \, d\mu_{d,v} - \int \lambda_{\alpha, v}(x)  \, d\mu_{d,v}\right |_v &\leq \int \left |\lambda_{M, v}(x) -\lambda_{\alpha, v}(x) \right|_vd\mu_{d,v} \nonumber \\
    & \leq 2\int_{|x-\alpha|_v \leq 1/M} |\log |x-\alpha|_v^{-1}|_v d\mu_{d,v} \nonumber \\
    & \leq O\left(\dfrac{\log M}{M}\right).
\end{align}
Using Proposition \ref{quant} for the map $f_{d,c}(z)=z^d+c$ with $d \geq 2$, we have
\begin{equation}\label{E3}
     \left| \frac{1}{|\mathcal{G}|} \sum_{x \in \mathcal{G}} \lambda_{M,v}(x) - \int \lambda_{M,v}(x) d\mu_{d,v} \right|_v \leq C_2  \left( \frac{\log|\mathcal{G}|}{|\mathcal{G}|} \right)^{1/2}\left(\log ^+|\alpha|_v+\dfrac{1}{\tau}\right).
\end{equation}
Now combining \eqref{E1}, \eqref{E2} and \eqref{E3} and choosing $\kappa < 1/4$ with $M = |\mathcal{G}|^{1/{2\kappa}}$, we get
\begin{align*}
\begin{split}
    &\left |\frac{1}{|\mathcal{G}|} \sum_{x \in \mathcal{G}} \lambda_{\alpha, v}(x) - \int \lambda_{\alpha, v}(x)  \, d\mu_{d,v}\right |_v 
   \leq \left |\frac{1}{|\mathcal{G}|} \sum_{x \in \mathcal{G}} \lambda_{M, v}( x) - \frac{1}{|\mathcal{G}|}\sum_{x \in \mathcal{G}} \lambda_{\alpha, v}(x)  \right |_v +\\ &  \left|\int \lambda_{M, v}(x)  \, d\mu_{d,v} - \int \lambda_{\alpha, v}(x)\, d\mu_{d,v}\right |_v +\left| \frac{1}{|\mathcal{G}|} \sum_{x \in \mathcal{G}} \lambda_{M,v}(x) - \int \lambda_{M,v}(x) d\mu_{d,v} \right|_v  \\
    &\leq \frac{|\mathcal{G}_2|}{|\mathcal{G}|}C_8+O\left(\dfrac{\log M}{M}\right)+C_2  \left( \frac{\log|\mathcal{G}|}{|\mathcal{G}|} \right)^{1/2}\left(\log ^+|\alpha|_v+\dfrac{1}{\tau}\right)  \\
    &\leq \frac{C_8}{|\mathcal{G}|}+O\left(\dfrac{\sqrt{\log |\mathcal{G}|}}{\sqrt{|\mathcal{G}|}}\right)\left(\log ^+|\alpha|_v+\dfrac{1}{\tau}\right).\\
    \end{split}
\end{align*}
In the third inequality, first term is bounded since $|\mathcal{G}_2| < 1$ and the sum of the other two terms is $O\left(\frac{\sqrt{\log |\mathcal{G}|}}{\sqrt{|\mathcal{G}|}}\right)\left(\log ^+|\alpha|_v+\dfrac{1}{\tau}\right)$. This completes the proof of Proposition \ref{Propalpha}.
\end{proof}
Now we are ready to prove main results.
\subsection{Proof of Theorem \ref{thm2}}
Suppose that $D$ is a positive integer and we consider a finite extension $L$ of $K$ of degree at most $D$. Let $S_L$ be the set of primes in $L$ lying over the primes in $S$. For non-archimedean $w \in S_L$ we have $\log|x - \alpha|^{-1}_w<\log\left(\dfrac{1}{r}\right)$ by Proposition \ref{prop3.2}  and for archimedean $w\in S_L$, by Proposition \ref{prop3.1}, $$\max_{x \in \mathcal{G}} \log |x - \alpha|_w^{-1}<C_6 \left(h(\alpha)+ \frac{d}{d-1}\right)|\mathcal{G}|^{8+\epsilon}.$$ Then from Proposition \ref{Propalpha}, such that for any $w\in S_L$, we have    
\begin{equation}\label{eq4.1}
	\left|\frac{1}{|\mathcal{G}|}\sum_{x \in \mathcal{G}} \lambda_{\alpha, w}(x)-\int_{\mathcal{M}_{d,w}} \lambda_{\alpha, w}(x) d\mu_{d, w}\right|_w \le \frac{C_8}{|\mathcal{G}|}+O\left(\dfrac{\sqrt{\log |\mathcal{G}|}}{\sqrt{|\mathcal{G}|}}\right)\left(\log ^+|\alpha|_w+\dfrac{1}{\tau}\right).
\end{equation}
Using $C_8<C_6 \left(h(\alpha)+ \frac{d}{d -1}\right)|\mathcal{G}|^{8+\epsilon}$ and for a suitable constant $C_9$, \eqref{eq4.1} can be written as 
\begin{equation} \label{sequi}
    \left|\frac{1}{|\mathcal{G}|}\sum_{x \in \mathcal{G}} \lambda_{\alpha, w}(x)-\int_{\mathcal{M}_{d,w}} \lambda_{\alpha, w}(x) d\mu_{d, w}\right|_w \le \frac{C_9}{|\mathcal{G}|^{1/2}}\left(h(\alpha)+\log ^+|\alpha|_w+ \dfrac{1}{\tau}+\dfrac{d}{d-1}\right).
\end{equation}
Suppose that $\mathrm{Gal}(\overline{K}/L)$-orbit of post-critically finite parameters satisfies $|\mathcal{G}|>C_{10}$ where $C_{10}:=C_{10}(D,|S_L|, d)$ is large enough. By taking the summation of $w \in S_L$ in \eqref{sequi}, we obtain 
\begin{equation}\label{bound}
    \sum_{w \in S_L}\left|\frac{1}{|\mathcal{G}|}\sum_{x \in \mathcal{G}}N_w\lambda_{\alpha, w}(x)- \int_{\mathcal{M}_{d,w}} N_w\lambda_{\alpha, w}(x) d\mu_{d, w}\right|_w \leq \frac{h(\alpha)+\dfrac{d}{d-1} + \dfrac{1}{\tau}}{ND^{2.5}}.
\end{equation}
By Proposition \ref{propadelic}, we get 
\begin{align*}
	\big|h_{L}(\mathcal{G})-\big<{\mathcal{L}},\mathcal{L}_{d,c}\big>\big|_w &\le C_{5} \Big( \frac{\log |\mathcal{G}|^{1/2}}{|\mathcal{G}|^{1/2}}\Big)\left(\log ^+|\alpha|_w+\dfrac{1}{\tau}\right)\\ & < \frac{1}{D^{2}}\left(\log ^+|\alpha|_w+\dfrac{1}{\tau}\right).
\end{align*}
As $x$ is $S_L$-integral relative to $\alpha$, we have $\lambda_{\alpha, w}(x)=0$ for all $w \notin S_L$ and $x \in \mathcal{G}$.
Therefore, 
\begin{align}\label{al4.10}
	\frac{1}{|\mathcal{G}|}\sum_{w \in M_{L}}\sum_{x \in \mathcal{G}}N_w\lambda_{\alpha, w}(x)-&\sum_{w \in M_{L}} \int_{\mathcal{M}_{d,w}} N_w\lambda_{\alpha, w}(x) d\mu_{d, w} \nonumber \\  &=h_{L}(\mathcal{G})-\langle L,L_{f_{d,c}}\rangle +h_{d,c}(\alpha) \nonumber \\ 
	& \ge -\frac{1}{D^2}\left(\log ^+|\alpha|_w+\dfrac{1}{\tau}\right)+h_{d,c}(\alpha).
\end{align}
Putting \eqref{al4.10} in \eqref{bound} and then simplifying, we get
\begin{equation}
    h_{d,c}(\alpha)-\frac{h(\alpha)}{ND^{2.5}}\leq \frac{\dfrac{d}{d-1} + \dfrac{1}{\tau}}{ND^{1.5}}+\frac{\log ^+|\alpha|_w+\dfrac{1}{\tau}+1}{D^2}.
\end{equation}  
Since $|h_{d,c}(\alpha)-h(\alpha)| < C_{11}$, where $C_{11}$ is an absolute constant depending on $f_{d,c}$, we have 
\begin{equation*}
    h(\alpha)\left(1-\frac{1}{ND^{2.5}}\right)\leq C_{11} +\frac{\dfrac{d}{d-1}+ \dfrac{1}{\tau}}{ND^{1.5}}+\frac{\log ^+|\alpha|_w+\dfrac{1}{\tau}}{D^2}.
\end{equation*}
If $h(\alpha) \geq 1$, then for sufficiently large $N$ we obtain 
\begin{equation*}
    D^2 \leq \dfrac{\log ^+|\alpha|_w+\dfrac{1}{\tau}}{1-C_{11}},
\end{equation*}
which is a contradiction for large $D$.
If $h(\alpha) <1$, then by a result of Dobrolowski \cite{dob}, we have
$h(\alpha) \geq \dfrac{C_{12}}{D(\log D)^3}$ for some constant $C_{12}>0$. Consequently, we have
\begin{align*}
   \frac{C_{12}}{D(\log D)^3}\left(1-\frac{1}{ND^{2.5}}\right)
   &\leq C_{11}+\frac{1}{ND^{1.5}}\left(\dfrac{d}{d-1} + \dfrac{1}{\tau}\right)+\frac{\log ^+|\alpha|_w+\dfrac{1}{\tau}}{D^2}.
\end{align*}
Then after simplifications, we deduce that
\begin{equation*}
\left (\log ^+|\alpha|_w+\dfrac{1}{\tau} \right) (\log D)^3 \geq D, 
\end{equation*}
which is impossible for large $D \geq 2$.
Hence, $|\mathcal{G}| < C_{10}$. This completes the proof of Theorem \ref{thm2}. \qed

\subsection{Proof of Theorem \ref{thm02}}
Suppose that $D$ is a positive integer and we consider a finite extension $L$ of $K$ of degree at most $D$. Let $S_L$ be the set of primes in $L$ lying over the primes in $S$. By Corollary \ref{cor3.3}, for each non-archimedean place $w$ of $K$, there exists at most one post-critically finite parameters $x \in \overline{\Q}$ such that $\log|x-\alpha|_v^{-1} > \log\dfrac{1}{r}$. We let $S_{\mbox{fin}}$ denote the subset of $S$ consisting exactly all non-archimedean places. Hence, up to at most $|S_{\mathrm{fin}}|$ exceptional $\mathrm{Gal}(\overline{K}/L)$-orbits, we assume that for every $x \in \mathcal{G}$ and every non-archimedean place $w$,
\begin{equation*}
\log |x-\alpha|_w^{-1} \le \log\dfrac{1}{r}.
\end{equation*}
Similarly, by Proposition \ref{archprop}, up to at most $|S \setminus S_\mathrm{fin}|$ exceptional $\text{Gal}(\overline{K}/L)$-orbits, we assume that for every $x \in \mathcal{G}$ and every archimedean place $w$, 
\begin{equation} \label{archbound}
     \max_{x \in \mathcal{G}} \log |x-\alpha|_w^{-1} \leq C_6  \left(h(\alpha)+ \frac{d}{d-1}\right)|\mathcal{G}|^{8+\epsilon}.
\end{equation}
Thus applying Proposition \ref{prop3.1} and Proposition \ref{prop3.2}, we assume \ref{archbound} for $w \in S_L$. We have from \eqref{sequi},
\begin{equation*}
 \left|\frac{1}{|\mathcal{G}|}\sum_{x \in \mathcal{G}} \lambda_{\alpha, w}(x)-\int \lambda_{\alpha, w}(x) d\mu_{d, w}\right|_w \le \frac{C_9}{|\mathcal{G}|^{1/2}}\left(h(\alpha)+\log ^+|\alpha|_w+ \dfrac{1}{\tau}+\dfrac{d}{d-1}\right).
\end{equation*} 
Summing up over all places in $S_L$ and for any constant $N > 0$ assuming that $ |\mathcal{G}| > C_{13}N|S_L|^3 D^8$ for some suitable constant $C_{13}$ gives,
\begin{align*}
    \sum_{w \in S_L}\Big|\frac{1}{|\mathcal{G}|}\sum_{x \in \mathcal{G}}N_w\lambda_{\alpha, w}(x)- \int N_w\lambda_{\alpha, w}(x) d\mu_{d, w}\Big|_w &\le \frac{h(\alpha) +\dfrac{1}{d-1}+ \dfrac{1}{\tau}}{ND^{4}}+ \frac{h(\alpha)}{N|S_L|D^{4}}\\ &\le \frac{h(\alpha) +\dfrac{d}{d-1}+ \dfrac{1}{\tau}}{ND^{4}}.
\end{align*}
Again by a similar argument used in the proof of Theorem \ref{thm2} leads to the following conclusion
\begin{equation} \label{conclusion}
 h(\alpha)\left(1-\frac{1}{ND^{4}}\right)\leq C_{11}+ \frac{\dfrac{d}{d-1}+ \dfrac{1}{\tau}}{ND^{4}}+\frac{\log ^+|\alpha|_w+\dfrac{1}{\tau}}{D^2}.
\end{equation} 
If $h(\alpha) \geq 1$, then for sufficiently large $N$ we obtain, 
\begin{equation*}
    D^2 \leq \dfrac{\log ^+|\alpha|_w+\dfrac{1}{\tau}}{1-C_{11}},
\end{equation*}
which is a contradiction for large $D$.
If $h(\alpha) <1$, then by a result of Dobrowolski \cite{dob}, we know $h(\alpha) > O\left(\dfrac{1}{D^{3/2}}\right)$. Then from \eqref{conclusion}, we will obtain
\begin{equation*}
    D^{0.5} < C_{14} \left( \log ^+|\alpha|_w+\dfrac{1}{\tau}\right),
\end{equation*}
which is a contradiction for large $D$. 
This completes the proof of Theorem \ref{thm02}.

%\textbf{Acknowledgments}
%The author thanks Prof. Su-Ion Ih for his useful comments in the preliminary version of this paper.

\end{document}